\documentclass[12pt]{amsart}
\usepackage{amsmath,amsthm,amsfonts,amssymb,times,latexsym,mathabx,url}

\newtheorem{theorem}{Theorem}[section]
\newtheorem{prop}[theorem]{Proposition}
\newtheorem{lem}[theorem]{Lemma}

\newtheorem{cor}[theorem]{Corollary}

\numberwithin{equation}{section}

\renewcommand{\a}{\alpha}
\newcommand{\ba}{\bar{\a}}

\renewcommand{\d}{\delta}
\newcommand{\e}{\varepsilon}
\renewcommand{\l}{\lambda}
\renewcommand{\o}{\omega}

\newcommand{\E}{\mathcal{E}}

\newcommand{\M}{\mathcal{M}}
\newcommand{\N}{\mathbb{N}}
\renewcommand{\P}{\mathcal{P}}

\renewcommand{\S}{\mathcal{S}}

\renewcommand{\leq}{\leqslant}
\renewcommand{\geq}{\geqslant}
\newcommand{\lcm}{\operatorname{lcm}}

\newcommand{\m}{\mathbf m}
\newcommand{\p}{\mathbf p}
\renewcommand{\l}{\mathbf l}

\makeatletter
\renewcommand{\pmod}[1]{\allowbreak\mkern7mu({\operator@font mod}\,\,#1)}
\makeatother

\begin{document}

\title{Moments of the shifted prime divisor function} 
\author{Mikhail R. Gabdullin}

\address{Department of mathematics, 1409 West Green Street, University of Illinois at Urbana-Champaign, Urbana, IL 61801, USA} 
\email{mikhailg@illinois.edu, gabdullin.mikhail@yandex.ru}

\begin{abstract}
Let $\o^*(n) = \{d|n: d=p-1, \mbox{$p$ is a prime}\}$. We show that, for each integer $k\geq2$,
$$
\sum_{n\leq x}\o^*(n)^k \asymp x(\log x)^{2^k-k-1}, 
$$ 
where the implied constant may depend on $k$ only. This confirms a recent conjecture of Fan and Pomerance. Our proof uses a combinatorial identity for the least common multiple, viewed as a multiplicative analogue of the inclusion-exclusion principle, along with analytic tools from number theory.
\end{abstract}



\date{\today}
\maketitle	
	
\section{Introduction}	 

Let $\tau(n)=\sum_{d|n}1$ be the divisor function. The behavior of moments of $\tau$ is well understood: it is well known that, for any fixed positive integer $k$,  	
\begin{equation}\label{1.1} 
\sum_{n\leq x}\tau(n)^k \sim c_kx(\log x)^{2^k-1}
\end{equation}
for some constant $c_k>0$ (see, for instance, \cite{W} and \cite{LT}). Several recent papers have studied the following analogue of $\tau$. For a positive integer $n$, let 
$$	
\o^*(n) = \{d|n: d=p-1, \mbox{$p$ is a prime} \},
$$
which can be called the `shifted prime divisor' function. This function was first studied by Prachar and Erd\H{o}s \cite{Erd}, \cite{Pr}. It is easy to see that 
\begin{equation}\label{1.2}
\sum_{n\leq x}\o^*(n)=x\log\log x+O(x), 
\end{equation}
similar to the average value of $\o(n)$, the number of prime divisors of $n$. However, for $k\geq2$ the behavior is very different: we have 
$$
\sum_{n\leq x}\o^*(n)^2 \asymp x\log x
$$ 
(see the papers \cite{Murty}, \cite{Ding}, and also \cite{DGZ}), and it was shown recently by  Fan and Pomerance \cite{FP} that
$$
\sum_{n\leq x}\o^*(n)^3 \asymp x(\log x)^4,
$$ 
who also conjectured that, for any fixed integer $k\geq2$,
\begin{equation}\label{1.3} 
\sum_{n\leq x}\o^*(n)^k \asymp x(\log x)^{2^k-k-1}.	
\end{equation}
Heuristically, the right side here is just the right side of (\ref{1.1}) divided by $(\log x)^k$, which is reasonable to expect since the probability that a given divisor is prime is about $1/\log x$. Currently, (\ref{1.3}) is known for $k=2,3$ only, and there are not even any known nontrivial upper bounds for any $k>3$ (a trivial bound follows from $\o^*(n)\leq \tau(n)$).

The main goal of the present paper is to establish the above conjecture (\ref{1.3}).

\begin{theorem}\label{th1} 
We have
$$
\sum_{n\leq x}\o^*(n)^k \asymp x(\log x)^{2^k-k-1}.
$$	
\end{theorem}

As was mentioned in \cite{FP}, showing (\ref{1.2}) can be reduced to computing the quantities
$$
P_k(x)=\#\big\{(p_1,\ldots,p_k): \lcm\{p_1-1,\ldots,p_k-1\}\leq x\big\},
$$
where all $p_j$ are assumed to be primes. More specifically, 
\begin{equation}\label{1.4}
P_k(x) \asymp x(\log x)^{2^k-k-2}	
\end{equation}
would imply (\ref{1.1}) (see Section 2 for details), and we will be working with $P_k(x)$. To show that the lower bound in (\ref{1.4}) holds, we use some combinatorial identity for the least common multiple (see Lemma \ref{lem2.1}), which can be viewed as a multiplicative version of the inclusion-exclusion principle. To prove the upper bound, we generalize the approach given in \cite{FP}, which exploits another combinatorial decomposition of the least common multiple together with the upper sieve.

Section \ref{sec2} contains an outline of the proof as well as some auxiliary statements we rely on. The lower bound in Theorem \ref{th1} is established in Section \ref{sec3}, and the upper bound is proved in Section \ref{sec4}. Finally, 
in Section \ref{sec5} we provide some estimates for the distribution 
$$
N(x,y)=\#\{n\leq x: \o^*(n)\geq y\}
$$
of $\o^*$, which was addressed in \cite{FP} (we briefly mention that the related function $\#\{n\leq x: \exists \mbox{\small{ a prime $p>y$ with }} p-1|n \}$ was studied recently in \cite{F} and \cite{MPP}).
 
\medskip 

\textbf{Acknowledgements.} The author is grateful to Alisa Sedunova for introducing him to this topic and, along with Kevin Ford, for useful discussions.

\section{Outline and notation} \label{sec2} 

Let a positive integer $k\geq2$ be fixed. We use Vinogradov’s $\ll$ notation: both $F \ll G$ and $F\ll O(G)$ mean that there exists a constant $C>0$ such that $|F| \leq CG$. We write $F \asymp G$ if $G \ll F \ll G$; all the implied constants depend at most on $k$.  We denote the least common multiple of positive integers $d_1,\ldots,d_k$ by $[d_1,...,d_k]$. The letters $p$ and $r$ always denote prime numbers. As is standard, $\varphi$ is Euler's totient function and $\log_2x$ and $\log_3x$ stand for $\log\log x$ and $\log\log\log x$, respectively.

As mentioned in the introduction, we will work with
$$
P_k(x)=\#\{(p_1,\ldots,p_k): [p_1-1,\ldots,p_k-1]\leq x\}.
$$
Since
\begin{align*}
&\sum_{n\leq x}\o^*(n)^k=\sum_{\substack{p_1,...,p_k: \\ [p_1-1,\ldots,p_k-1] \leq x}}\left\lfloor\frac{x}{[p_1-1,\ldots,p_k-1]}\right\rfloor \asymp x\sum_{\substack{p_1,\ldots,p_k: \\ [p_1-1,\ldots,p_k-1] \leq x}}\frac{1}{[p_1-1,\ldots,p_k-1]}=\\
&x\sum_{m\leq x}\frac{\#\{(p_1,\ldots,p_k): [p_1-1,\ldots,p_k-1]=m\}}{m}=P_k(x)+x\int_1^x\frac{P_k(y)}{y^2}dy,
\end{align*}
to prove Theorem \ref{th1}, it is enough to show that, for all large $x$,
\begin{equation}\label{2.1}
P_k(x) \asymp x(\log x)^{2^k-k-2}.	
\end{equation}

Let $[k]=\{1,\ldots,k\}$ and, for a nonempty subset $S\subseteq[k]$ and numbers $t_1,\ldots,t_k$, define
\begin{equation}\label{2.2}
t_S=\gcd\{t_i: i\in S\}.
\end{equation}
We will also write $t_{ij}$ for $t_{\{i,j\}}$ for $i\neq j$. The following statement is the heart of the proof of the lower bound in (\ref{2.1}).

\begin{lem}\label{lem2.1}
Let $k\geq2$. For any positive integers $t_1,\ldots,t_k$,
\begin{equation}\label{2.3}
[t_1,\ldots,t_k]=\frac{\prod_{\substack{1\leq i\leq k\\ i \mbox{\tiny \, \textup{odd}}}}{\prod_{B\subseteq [k], |B|=i}t_B}}{\prod_{\substack{2\leq j\leq k\\ j \mbox{\tiny \, \textup{even}}}}{\prod_{C\subseteq [k], |C|=j}t_C}}.
\end{equation}	
\end{lem}

\begin{proof}
This is a version of inclusion-exclusion principle. Fix a prime $p$ and, for each $i$, let $p^{\a_i}|| t_i$; without loss of generality we may assume that $\a_1\leq \a_2\leq\ldots\leq \a_k$ (since both sides of (\ref{2.3}) do not change under any permutation of $t_i$). Clearly, for $S\subseteq[k]$, we have $p^{\a_{\min S}}||t_S$.Then $p$ appears with the exponent $\a_k$ in the left side, and with the exponent 
$$
H=\sum_{\substack{S\subseteq[k]\\ |S|\geq1}}(-1)^{|S|-1}\a_{\min S}
$$ 
in the right side. If $\min S=i$, then $|S|\leq k-i+1$ and there are ${k-i \choose j-1}$ choices of such $S$ with $|S|=j$. Thus, 
$$
H=\a_k+\sum_{i=1}^{k-1}\a_i\sum_{j=1}^{k-i+1}(-1)^{j-1}{k-i \choose j-1}=\a_k+\sum_{i=1}^{k-1}\a_i\sum_{j=0}^{k-i}(-1)^j{k-i \choose j}=\a_k.
$$
The claim follows.
\end{proof}

Note that values of $t_{ij}$ give us values of $t_S$ for all $S\subseteq[k]$ with $|S|\geq2$ (because $t_S=\gcd\{t_{ij}: i,j\in S\}$). Our main idea is that once we know size of each $p_i$ and also know all pairwise greatest common divisors $\gcd(p_i-1,p_j-1), 1\leq i< j\leq k$, we control size of the least common multiple $[p_1-1,\ldots,p_k-1]$. 

Since the cases $k=2$ and $k=3$ of Theorem \ref{th1} were established in the previous works, we will focus on the case $k\geq4$. We describe the case $k=4$ as the model situation. Let $E$ be a large enough absolute constant. We start with the multidimensional dyadic decomposition, which gives 	
$$	
P_4(x)=\sum_{\substack{l_i\leq \log x+1 \\ 1\leq i\leq4}}\#\{(p_1,\ldots,p_4): E^{l_i-1}<p_i-1\leq E^{l_i}, [p_1-1,\ldots,p_4-1]\leq x\}.
$$
Assume that $E^{l_i-1}<p_i-1\leq E^{l_i}$ for all $i$. Let us ignore that $p_i-1$ are shifted primes for a moment. Denote $t_i=p_i-1$ for short, and recall the notation (\ref{2.2}). Lemma \ref{lem2.1} gives us
\begin{equation*}
[t_1,\ldots,t_4]=\frac{\prod_{i=1}^4t_i\cdot \prod_{|B|=3}t_{B}}{\prod_{|A|=2}t_A \cdot t_{[4]}},
\end{equation*}
where $A$ and $B$ go over subsets of $[4]=\{1,2,3,4\}$. Then $[t_1,...,t_4]\leq x$ is guaranteed by
\begin{equation}\label{2.4}
\frac{\prod_{|A|=2}t_A\cdot t_{[4]}}{\prod_{|B|=3}t_B} \geq E^{l_1+l_2+l_3+l_4}x^{-1}.
\end{equation}
Let numbers $t_i\in(E^{l_i-1}, E^{l_i}]$, $1\leq i\leq 4$, be such that $\gcd(t_i,t_j)=m_{ij}$ for all $1\leq i<j\leq 4$. For any sets $A_1,A_2,B_1,B_2$ with $|A_1|=|A_2|=|B_1|=|B_2|=2$, we must have 
\begin{equation}\label{2.5}
\gcd(m_{A_1},m_{B_1})=\gcd(m_{A_2},m_{B_2})
\end{equation}
whenever $A_1\cup B_1=A_2\cup B_2$, since $t_{A_1\cup B_1}=t_{A_2\cup B_2}$. Let a six-tuple 
$$
\m=\big(m_A: A\subseteq[4], |A|=2\big)
$$
satisfy (\ref{2.5}) and set
\begin{equation*}
M_i=M_i(\m):=\lcm\{m_A: A\ni i, |A|=2\}
\end{equation*}
for each $i$. Denote similarly 
$$
m_S:=\gcd\{m_{ij}: i,j\in S\};
$$ 
let also $\l=(l_1,...,l_4)$. Then
\begin{equation*}
P_4(x)\geq \sum_{\m}\sum_{\l}\P(\l,\m),	
\end{equation*}
where $\l$ is such that $l_i\leq \log x$ for each $i$,
$$
\P(\l,\m)=\#\{(p_1,\ldots,p_4): E^{l_{i-1}}<p_i\leq E^{l_i}, \gcd(p_i-1,p_j-1)=m_{ij}\} 
$$
and $\m$ satisfy (\ref{2.5}) and 
\begin{equation*}
\frac{\prod_{|A|=2}m_A\cdot m_{[4]}}{\prod_{|B|=3}m_B}\geq E^{l_1+l_2+l_3+l_4}x^{-1}	
\end{equation*}
(which is just (\ref{2.4})). Clearly, if $(p_1,\ldots,p_4)$ is counted in $\P(\l,\m)$, then each $p_i-1$ is divisible by $M_i$. Though the reverse is generally not true, one can impose additional restrictions on $\m$ and $\l$ and use Gallagher's prime number theorem to obtain a lower bound on $\P(\l,\m)$ for a large enough set of pairs $(\m,\l)$. After that, one only needs to deal with special multiplicative functions, for which standard analytic number theory tools are available. This strategy will allow us to establish the lower bound for $P_4$ of correct order, and the case of general $k$ is essentially identical.  

The same approach, together with the Brun-Titchmarsh inequality, yields an upper bound for $P_k(x)$. However, if we simply the number of primes $p_i\leq E^{l_{i}}$ that are equal to $1$ modulo $M_i$ by $E^{l_i}/\big(\varphi(M_i)\log(2E^{l_i}/M_i)\big)$, we ultimately obtain an upper bound of the form
$$
P_k(x) \ll x(\log x)^{2^k-k-1}(\log_2x)^k,
$$ 
because we lose something when $M_i$ is close to $E^{l_i}$. Instead of this, we generalize the proof of the upper bound $P_3(x)\ll (\log x)^3$ given in \cite{FP}. This requires another notation. Let $t_1,\ldots,t_k$ be any positive integers, as before. All subsets of $[k]$ considered below are assumed to be nonempty. Let $u_{[k]}=t_{[k]}$ and, once the numbers $u_S$ with $S\subseteq[k]$, $|S|=i\geq2$, are defined, set
\begin{equation}\label{2.6}
u_S=t_S\prod_{\substack{T\supset S\\|T|\geq i}}u_{T}^{-1}
\end{equation}
for each $S\subseteq[k]$ with $|S|=i-1$. Then 
\begin{align*} 
& [t_1,\ldots,t_k]=\prod_{S\subseteq[k]}u_S, \\
& t_S=\prod_{T\supseteq S}u_T,
\end{align*}
and 
\begin{equation}\label{2.7}
\gcd\left(\prod_{S\supseteq S_1}u_S,\prod_{S\supseteq S_2}u_S\right)=\prod_{S\supseteq S_1\cup S_2}u_S
\end{equation}
for all $S_1,S_2\subseteq[k]$. Let us turn again to the case $k=4$. Let a quadruple $\p=(p_1,\ldots,p_4)$ be counted in $P_4(x)$ and denote $t_i=p_i-1$ for short for each $i=1,\ldots,4$. Then  
\begin{equation}\label{2.8}
\prod_{S\subseteq [4]}u_S(\p) \leq x. 
\end{equation}
We have
$$
P_4(x) \leq \sum_{T\subseteq[4]}P_{4,T}(x),
$$
where $P_{4,T}(x)$ counts the number of quadruples with $\max_{S}u_S(\p)=u_T(\p)$.
There are four cases depending on what the cardinality of $T$ is, but they are all similar. For instance, let us consider the case $T=T_0:=\{3,4\}$. Since there are $15$ factors in the left side of (\ref{2.8}), we have 
\begin{equation}\label{2.9}
\prod_{S\neq T_0}u_S(\p) \leq x^{14/15}.
\end{equation}
For a positive integer $l\geq2$, let 
\begin{equation}\label{2.10}
\tau_l(n)=\#\{(d_1,\ldots,d_l)\in \N^l: d_1\cdot\ldots\cdot d_l=n\}.
\end{equation}
Now we choose a prime $p_1$ and fix a factorization $p_1-1=\prod_{S\ni 1}u_S$; here we have eight factors in the right side and thus $\tau_8(p_1-1)$ such factorizations. We fix any of them such that (\ref{2.7}) holds for the set $\{u_S: S \ni 1\}$. Set $V_1=\prod_{S\ni 1,2} u_S$ and choose $p_2$ which is $\equiv 1\pmod{V_1}$; next, we fix a factorization $(p_2-1)/V_1=\prod_{S\ni2, S\notni 1}u_S$ (here there are $\tau_4((p_2-1)/V_1)$ possibilities) such that the set $\{u_S: \min S\leq 2\}$ satisfies (\ref{2.7}). Note that
\begin{equation}\label{2.11}
(p_1-1)\frac{p_2-1}{V_1} = \prod_{S:\, \min S\leq 2}u_S.
\end{equation}
Denote $G=\prod_{\min S\leq 2}u_S$. For $S$ with $\min S>2$ and $S\neq T_0$, we choose any positive integers $u_S$ such that (\ref{2.7}) holds for $u_S$, $S\neq T_0$, and
$$
\prod_{\min S>2, \, S\neq T_0}u_S\leq x^{14/15}G^{-1}.
$$ 
For $j=3,4$, let
$$
a_j=\prod_{S\ni j, \, S\ne T_0}u_S. 
$$ 
Then $u_{T_0}\leq x/\prod_{S\neq T_0}u_S$ must be such that $a_ju_{T_0}+1$ is a prime for $j=3,4$. Also, $\log (x/\prod_{S\neq T_0}u_S) \gg \log x$ by (\ref{2.10}). By Brun's sieve and (\ref{2.11}), we have at most
$$
\ll \frac{x}{\prod_{S\neq T_0}u_S\cdot (\log x)^2}\left(\frac{\Delta}{\varphi(\Delta)}\right)^2= \prod_{i=1}^{2}\frac{1}{(p_i-1)/V_{i-1}}\cdot\frac{x}{\prod_{\substack{S\neq T_0\\\min S>2}}u_S\cdot (\log x)^2}\left(\frac{\Delta}{\varphi(\Delta)}\right)^2
$$
options for $u_{T_0}$, where we set $V_0=1$ and
$$
\Delta=a_3a_4(a_3-a_4).
$$
Note that $a_3\neq a_4$ by (\ref{2.8}) applied with $S_1=\{3\}$ and $S_2=\{4\}$, so that $\Delta\neq0$. Let us write $u_i$ instead of $u_{\{i\}}$ for brevity. Then
$$
P_{4,{T_0}}(x) \ll 
\sum_{p_1\leq x}\frac{\tau_8(p_1-1)}{p_1-1} \sum_{p_2}\frac{\tau_4((p_2-1)/V_1)}{(p_2-1)/V_1}\sum_{u_3,u_4}\frac{x}{u_3u_4\cdot (\log x)^2}\left(\frac{\Delta}{\varphi(\Delta)}\right)^2 
$$
where the second summation is over $p_2\leq x$ with $p_2\equiv 1\pmod{V_1}$. The factor $(\Delta/\varphi(\Delta))^2$ can be shown to be $O(1)$ on average uniformly for choices of $u_S$ with $\min S\leq 2$, and thus, since both $u_3$ and $u_4$ are at most $x$, the innermost sum is $O(x)$ for any choice of $u_S$ with $\min S\leq 2$. So,
\begin{equation}\label{2.12} 
P_{4,T_0}(x) \ll x\sum_{p_1}\frac{\tau_8(p_1-1)}{p_1-1} \sum_{p_2}\frac{\tau_4((p_2-1)/V_1)}{(p_2-1)/V_1}.
\end{equation}
To deal with these sums, we will need a result from \cite{Pol} about average value of multiplicative functions over the shifted primes (see Lemma \ref{lem4.1} below). Roughly, the idea is the following: the function $\tau_l(n)$ is multiplicative and $\tau_l(p)=l$ for every prime $p$ by (\ref{2.10}); by classical estimates for average value of multiplicative functions (for instance, Lemma \ref{lem3.5} below) we have
$$
\sum_{n\leq y}\frac{\tau_l(n)}{n} \asymp (\log y)^l.
$$
Shifted primes are known to have a similar anatomy to that of typical positive integers; then (since there are about $y/\log y$ summands below)
$$
\sum_{p\leq y}\frac{\tau_l(p-1)}{p-1} \asymp (\log y)^{l-1}.
$$
Putting this into (\ref{2.12}) (and ignoring the appearance of $V_1$ for now) would give us
$$
P_{4,T_0}(x) \ll x(\log x)^{10}.
$$
One can handle all $P_{4,T}(x)$ in this manner, and then the bound $P_4(x)\ll x(\log x)^{10}$ will follow. The case of general $k$ is more technical but essentially identical. We will turn to the details for both lower and upper bounds in the next sections. 


\section{The lower bound} \label{sec3}

We will make use of lower bounds for the number of primes in arithmetic progressions and thus need to address the possibility of an exceptional Siegel zero. The next statement is a corollary of Landau-Page Theorem.

\begin{theorem}\label{th3.1}
Let $x$ be large. Then there is $Q(x)$ which is either equal to
$1$ or is a prime of size $Q\gg \log_2 x$ such that $L(\sigma +it, \chi)\neq 0$ for any character $\chi$ of modulus at most $x$ and coprime to $Q(x)$ whenever 	
$$
1-\sigma \leq \frac{c_0}{\log(x(1+|t|))},
$$
where $c_0>0$ is an absolute constant. 
\end{theorem}	

\begin{proof}
See, for instance, \cite[Corollary 1]{FMT}.	
\end{proof}

\begin{theorem}[Gallagher's prime number theorem] \label{th3.2}
Let $q$ be a positive integer and suppose that $L(s,\chi) \neq 0$ for all characters $\chi$ of modulus $q$ and $s=\sigma+it$ with $1-\sigma \leq \frac{c}
{\log(q(1+|t|))}$ and some constant $c>0$. Then there is a constant $D=D(c)\geq 1$ depending only on $c$ such that
$$
\#\{p \mbox{ prime} : p \leq x; p \equiv a \pmod q\} \gg 
\frac{x}{\varphi(q) \log x}
$$
for all $a$ with $\gcd(a, q) = 1$ and $x \geq q^D$.		
\end{theorem}

\begin{proof}
See \cite{Gallagher} and also \cite[Lemma 2]{Maier}.
\end{proof}


Fix $k\geq4$. Set 
\begin{equation}\label{3.1}
\d=\min\bigg\{\frac{1}{D\big(c_0/(2k)\big)},\frac{1}{100}\bigg\}	
\end{equation}
and, for each $i=0,...,k$, let 
\begin{equation}\label{3.2}
x_i=x^{\d^{2(k-i)}}
\end{equation}
and
$$
Q_0=Q(x_0)\cdot\ldots\cdot Q(x_{k-1}).
$$ 
From the above theorems, we can deduce

\begin{cor}\label{cor3.3}
Let $x$ be large, $1\leq i\leq k$ and $y\geq x_i^{1/2}$. Let $E$ be large enough absolute constant, and suppose that $q$ and $q'$ are coprime with $q$ square-free, $q'$ odd, $x_{i-1}^{1/(2k)}<q^2q'\leq x_{i-1}$ and $\gcd(qq',Q_0)=1.$ Then, for any fixed integer $s\geq0$, 
\begin{equation}\label{3.3}
\sum_{\substack{p\in (y/E, \, y]: \\ \gcd(p-1, \, q^2q'Q_0)=q}}\left(\frac{\varphi(p-1)}{p-1}\right)^s \gg \frac{y}{q\log x}\frac{\varphi(q')}{q'}\left(\frac{\varphi(q)}{q}\right)^s.
\end{equation}
\end{cor}

\begin{proof} Let $\chi$ be a character of modulus $q^2q'$. Since $q^2q'$ is at most $x_{i-1}$ and coprime to $Q_{i-1}$, Theorem \ref{th3.1} implies that $L(\sigma+it,\chi)\neq0$ whenever 
$$
1-\sigma \leq \frac{c_0}{\log\big(x_{i-1}(|t|+1)\big)}.
$$
Because of the lower bound on $q^2q'$, the above inequality is guaranteed if
$$
1-\sigma \leq \frac{c_0/(2k)}{\log\big(q^2q'(|t|+1)\big)}.
$$	
Let $\gcd(a,q^2q')=1$. Note that $q^2q'\leq y^{\d}$. By the definition (\ref{3.1}) of $\d$ and Theorem \ref{th3.2},
$$
\#\{p\leq y: p\equiv a\pmod{q^2q'}\} \gg \frac{y}{\varphi(q^2q')\log x}
$$
and also
$$
\#\{p\leq y/E: p\equiv a\pmod{q^2q'}\} \ll \frac{y/E}{\varphi(q^2q')\log x},
$$
by the Brun-Titchmarsh inequality. Therefore,
\begin{equation}\label{3.4} 
\#\{p\in (y/E,y]: p\equiv a\pmod{q^2q'}\} \gg \frac{y}{\varphi(q^2q')\log x},	
\end{equation}	 
provided that $E$ is large enough.	
	
Recall that $r$ stands for prime numbers. By the Chinese Remainder theorem, there are exactly 
$$
\varphi_1(q'):=q'\prod_{r|q'}\frac{r-2}{r}\asymp q'(\varphi(q')/q')^2=\frac{\varphi(q')^2}{q'}
$$ 
residue classes $a'$ modulo $q'$ with $\gcd(a',q')=\gcd(a'-1,q')=1$, and exactly $\varphi(q)$ residues classes $a''$ modulo $q^2$ with $\gcd(a''-1,q^2)=q$. Thus, there are $\asymp \varphi(q)\varphi(q')^2/q'$ residue classes $a$ modulo $q^2q'$ with $\gcd(a,q^2q')=1$ and $\gcd(a-1,q^2q')=q$. Summing (\ref{3.4}) over all such $a$, we see that
\begin{equation}\label{3.5} 
\#\{p\in (y/E,y]: \gcd(p-1,q^2q')=q\} \gg \frac{y}{q\log x}\frac{\varphi(q')}{q'}.
\end{equation}
If $Q_0=1$, we are done with the case $s=0$. Otherwise, $Q_0$ is the product of at most $k$ primes $Q_j'$, each of which is $\gg\log_2 x$. Then we can argue similarly and apply the Brun-Titchmarsh inequality again to get
$$
\#\{p\leq y: \gcd(p-1,q^2q')=q, \gcd(p-1,Q_j')=Q_j'\} \ll \frac{y}{q\varphi(Q_j')\log x}\frac{\varphi(q')}{q'}
$$
for any prime divisor $Q_j'$ of $Q_0$. Since $\varphi(Q_j')= Q_j'-1\gg \log_2 x$,
$$
\#\{p\leq y: \gcd(p-1,q^2q')=q, \gcd(p-1,Q_0)>1\} \ll \frac{y}{q\log x\log_2x}\frac{\varphi(q')}{q'},
$$
which, together with (\ref{3.5}), yields the claim (\ref{3.3}) for $s=0$. 

Now we turn to the case $s\geq1$. Since 
$$
\frac{\varphi\big((p-1)/q\big)}{(p-1)/q} \leq 1,
$$
we may assume that $s\geq2$. By (\ref{3.3}) for $s=0$ and H\"older's inequality,
$$
\frac{y}{q\log x}\frac{\varphi(q')}{q'}\ll \bigg(\sum_{\substack{p\in (y/E, \, y]: \\ \gcd(p-1, \, q^2q'Q_0)=q }}\left(\frac{\varphi(p-1)}{p-1}\right)^s\bigg)^{1/s}\bigg(\sum_{\substack{p\leq y: \\ \gcd(p-1, \, q^2q')=q }}\left(\frac{p-1}{\varphi(p-1)}\right)^{s/(s-1)}\bigg)^{(s-1)/s},
$$
so that (\ref{3.3}) will follow from the inequality
\begin{equation}\label{3.6}
S=S(y,q,q'):=\sum_{\substack{p\leq y: \\ \gcd(p-1, \, q^2q')=q }}\left(\frac{(p-1)/q}{\varphi((p-1)/q)}\right)^2 \ll \frac{y}{q\log x}\frac{\varphi(q')}{q'}.
\end{equation} 
For a prime $p$ involved in the above sum, we have 
$$
\frac{(p-1)/q}{\varphi((p-1)/q)}=\prod_{\substack{r|p-1\\ r\notdivides q}}\big(1-1/r\big)^{-1} \ll \prod_{\substack{r|p-1\\ r\notdivides q}}(1+1/r)=\sum_{\substack{d|p-1\\ \gcd(d,q)=1}}\frac{\mu^2(d)}{d},
$$
and, hence,
\begin{align*}
& S \ll \sum_{\substack{p\leq y: \\ \gcd(p-1, \, q^2q')=q }}\sum_{\substack{d_1,d_2|p-1\\ \gcd(d_1d_2,q)=1}}\frac{\mu^2(d_1)\mu^2(d_2)}{d_1d_2}\leq\\
& \sum_{\substack{d_1,d_2\leq y,\\ \gcd(d_1d_2,qq')=1}}\frac{1}{d_1d_2}\#\big\{p\leq y: p\equiv 1\pmod{[d_1,d_2]}, \gcd(p-1,q^2q')=q\big\}.	
\end{align*}
For fixed $d_1,d_2$, there are $\varphi(q)\varphi_1(q')$ residue classes $b$ modulo $[d_1,d_2,q^2q']$ for which $b\equiv 0\pmod{[d_1,d_2]}$ and $\gcd(b,q^2q')=q$. 
Thus, if $[d_1,d_2,q^2q']>y$, then trivially
$$
\#\big\{p\leq y: p\equiv 1\pmod{[d_1,d_2]}, \gcd(p-1,q^2q')=q\big\}\leq \varphi(q)\varphi_1(q')\leq y^{\d},
$$
and 
$$
\#\big\{p\leq y: p\equiv 1\pmod{[d_1,d_2]}, \gcd(p-1,q^2q')=q\big\}\ll 	\frac{y\varphi(q)\varphi_1(q')}{\varphi([d_1,d_2,q^2q']) \log(2y/[d_1,d_2,q^2q'])}
$$
by the Brun-Titchmarsh inequality otherwise. In addition, if $y^{1/2}<[d_1,d_2,q^2q']\leq y$, then the upper bound in the latter case is $\ll y^{0.51+\d}\leq y^{0.52}$ (since $\varphi(m)\gg m/\log_2m$). Note also that $\varphi([d_1,d_2,q^2q'])=\varphi([d_1,d_2])q\varphi(q)\varphi(q')$ and $\varphi_1(q')/\varphi(q')\asymp \varphi(q')/q'$. Using 
$$
\sum_{d_1,d_2\leq y}\frac{1}{d_1d_2} \ll (\log y)^2,
$$
and putting all together, we obtain
$$
S\ll \sum_{\substack{d_1,d_2: \\ [d_1,d_2,q^2q']\leq y^{1/2}}}\frac{1}{d_1d_2\varphi([d_1,d_2])}\frac{y\varphi(q')}{qq'\log x}+y^{0.53}.
$$
Now (\ref{3.6}) follows, since the series
\begin{equation*} 
\sum_{d_1,d_2=1}^{\infty}\frac{1}{d_1d_2\varphi([d_1,d_2])} 
\end{equation*}
converges and $qq'\leq y^{\d}\leq y^{1/100}$. This concludes the proof.
\end{proof}

Now we are ready to prove the lower bound in Theorem \ref{th1}. Assume that a ${k \choose 2}$-tuple \\$\m=(m_A: |A|=2)$ satisfies  
\begin{equation}\label{3.7}
\gcd(m_{A_1},m_{B_1})=\gcd(m_{A_2},m_{B_2})
\end{equation}
for any $A_1,A_2,B_1,B_2$ with $|A_1|=|A_2|=|B_1|=|B_2|=2$ and $A_1\cup B_1=A_2\cup B_2$, and let
$$
M_i=M_i(\m)=\lcm\{m_A: A\ni i, |A|=2\}.
$$
We claim that 
\begin{equation}\label{3.71}
\gcd(M_i,M_j)=m_{ij}
\end{equation}
for any $1\leq i<j\leq k$. Indeed, fix a prime $p$ and assume that, for each $|A|=2$, the number $\a_A=\a_A(p)\geq0$ is such that $p^{\a_A}||m_A$, and let us write $\a_{ij}$ instead of $\a_{\{i,j\}}$. Then (\ref{3.8}) is equivalent to 
\begin{equation}\label{3.72}
\min\Big\{\max_{A\ni i}\a_A, \max_{B\ni j}\a_B\Big\}=\a_{ij}
\end{equation}
for each $p$. Clearly, the left side is at least $\a_{ij}$. Now if $s_1,s_2$ are such that $\max_{A\ni i}\a_A=\a_{is_1}\leq \a_{js_2}=\max_{B\ni j}\a_B$, we have by (\ref{3.7})
$$
\a_{is_1}=\min\{\a_{is_1},\a_{js_2}\}= \min\{\a_{ij},\a_{s_1s_2}\}\leq \a_{ij},
$$
which gives (\ref{3.72}), and (\ref{3.71}) follows. In particular, we see that a ${k \choose 2}$-tuple $\m=\{m_{ij}\}$ occurs as $\{\gcd(m_i,m_j)\}$ for some numbers $m_1,\ldots, m_k$ if and only it obeys (\ref{3.7}).

For a vector $\l=(l_1,\ldots,l_k)$ and $\m$ with (\ref{3.7}), define 
\begin{equation}\label{3.8}
\P(\l,\m)=\#\{(p_1,\ldots,p_k): E^{l_i-1}<p_i\leq E^{l_i}, \gcd(p_i-1,p_j-1)=m_{ij}\}.
\end{equation}
Assume that $(p_1,\ldots,p_k)$ is counted in $\P(\l,\m)$; by Lemma \ref{lem2.1}, the inequality
\begin{equation}\label{3.9}
\frac{\prod_{\substack{2\leq j\leq k\\ j \mbox{\tiny \, even}}}{\prod_{C\subseteq [k], |C|=j}m_C}}{\prod_{\substack{3\leq i\leq k\\ i \mbox{\tiny   \,odd}}}{\prod_{B\subseteq [k], |B|=i}m_B}} \geq E^{l_1+\ldots+l_k}x^{-1}
\end{equation}
guarantees that $[p_1-1,\ldots,p_k-1]\leq x$. Then
$$
P_k(x)\geq \sum_{\m}\sum_{\l}\P(\l,\m),
$$
where the summation is over $\m,\l$ which satisfy (\ref{3.7}) and (\ref{3.9}). In order to extract the lower bound for $P_k(x)$, we put the further restrictions on $\m$ and $\l$: we require that
\begin{align}
& \mu^2(m_A)=1, \, \gcd(m_A,Q_0)=1 	\mbox{ for all $|A|=2$};  \label{3.10} \\
&  x_0^{1/(4k)} \leq M_1 \leq \lcm\{m_A: |A|=2\} \leq x_0^{1/2}; \label{3.11}\\
& x_i^{1/2}<E^{l_i}\leq x_i^{2/3} \mbox{ for all $i=1,\ldots,k$}. \label{3.12}
\end{align}
(recall the definition (\ref{3.2}) of $x_i$). In particular, $l_i\asymp \log x$ for all $i$. Fix such a pair $(\m,\l)$. A $k$-tuple of primes $(p_1,\ldots,p_k)$ is counted in $\P(\l,\m)$ iff $E^{l_i-1}<p_i\leq E^{l_i}$ for all $i$ and
\begin{align*} 
& p_1\equiv 1\pmod {M_1}, \\
& p_2\equiv 1\pmod {M_2} \mbox{ and } \gcd\left(\frac{p_2-1}{m_{12}},\frac{p_1-1}{m_{12}}\right)=1,\\
& \ldots \\
& p_k\equiv 1\pmod {M_k} \mbox{ and } \gcd\left(\frac{p_k-1}{m_{ik}},\frac{p_i-1}{m_{ik}}\right)=1 \mbox{ for all $i=1,\ldots,k-1$}
\end{align*}
(note that all $M_i$ are square-free, are at most $x_0^{1/2}$ and $\gcd(M_i,M_j)=m_{ij}$ because of (\ref{3.7})). For fixed primes $p_1,\ldots,p_{i-1}$, denote
$$
F_1=1, \quad F_i=F_i(p_1,\ldots,p_{i-1}):=\lcm\bigg\{\frac{p_j-1}{m_{ji}}: j<i\bigg\}, i=2,\ldots,k.
$$
Then $\P(\l,\m)$ is at least the number of $k$-tuples $(p_1,\ldots,p_k)$ such that, for each $i=1,\ldots,k$, $E^{l_i-1}<p_i\leq E^{l_i}$, $\gcd(p_i-1,Q_0)=1$ and
$$
\gcd(p_i-1, M_i^2)=M_i \quad\mbox{ and }\quad \gcd(p_i-1,F_i)=1,
$$
(note that $\gcd(M_i,F_i)=1$ for all $i$). For fixed $p_1,\ldots,p_{k-1}$, by Corollary \ref{cor3.3} (with $i=k$, $y=E^{l_k}$, $q=M_k$, $q'=F_k$, and $s=0$) there are 
$$
\gg \frac{E^{l_k}}{M_k\log x}\prod_{i=1}^{k-1}\frac{\varphi(p_i-1)}{p_i-1}
$$
choices of $p_k\in (E^{l_k-1}, E^{l_k}]$, since crudely by (\ref{3.10}) and (\ref{3.11}), 
\begin{align}\label{3.13}
&M_k^2F_k\geq F_k \geq \frac{p_{k-1}-1}{m_{k-1,k}} \geq \frac{E^{l_{k-1}-1}-1}{m_{k-1,k}} \geq x_{k-1}^{1/3}, \\
& M_k^2F_k \leq x^{\d^{2k}}\prod_{i=1}^{k-1}(p_i-1)\leq x^{\d^{2k}+2(\d^{2k-2}+...+\d^2)/3}\leq x^{\d^2}=x_{k-1} \label{3.14} 
\end{align}
and also $E^{l_k}>x_k^{1/2}$. Thus,
$$
\P(\l,\m) \gg  
\frac{E^{l_k}}{M_k\log x}\sum_{p_1}\frac{\varphi(p_1-1)}{p_1-1}\ldots \sum_{p_{k-1}}\frac{\varphi(p_{k-1}-1)}{p_{k-1}-1}.
$$
Now, for fixed $p_1,\ldots,p_{k-2}$, we can apply Corollary \ref{cor3.3} with $i=k-1$, $y=E^{l_{k-1}}$, $q=M_{k-1}$, $q'=F_{k-1}$, and $s=1$ (since, similarly to (\ref{3.13}) and (\ref{3.14}), $x_{k-2}^{1/3}\leq M_{k-1}^2F_{k-1} \leq x_{k-2}$) to obtain
$$
\P(\l,\m) \gg  
\frac{E^{l_k+l_{k-1}}}{M_kM_{k-1}(\log x)^2}\frac{\varphi(M_{k-1})}{M_{k-1}}\sum_{p_1}\left(\frac{\varphi(p_1-1)}{p_1-1}\right)^2\ldots \sum_{p_{k-2}}\left(\frac{\varphi(p_{k-2}-1)}{p_{k-2}-1}\right)^2.
$$
Arguing similarly and applying Corollary \ref{cor3.3} consecutively $k-3$ more times, we see that
$$
\P(\l,\m) \gg  \frac{E^{\sum_{i=2}^kl_i}}{\left(\prod_{i=2}^kM_i\right)(\log x)^{k-1}}\left(\frac{\varphi(\prod_{i=2}^kM_i)}{\prod_{i=2}^kM_i}\right)^{k-2}\sum_{\substack{p_1\equiv 1\pmod{M_1}\\ E^{l_1-1}<p_1-1\leq E^{l_1}}}\left(\frac{\varphi(p_1-1)}{p_1-1}\right)^{k-1}. 
$$
Now the final application of Corollary \ref{cor3.3} to the sum over $p_1$ (with $i=1$, $y=E^{l_1}$, $q=M_1$, $q'=F_1=1$, $s=k-1$) gives 
\begin{equation}\label{3.15}
\P(\l,\m) \gg  \frac{E^{\sum_{i=1}^kl_i}}{\left(\prod_{i=1}^kM_i\right)(\log x)^k}\left(\frac{\varphi(\prod_{i=1}^kM_i)}{\prod_{i=1}^kM_i}\right)^{k-1}. 	
\end{equation}
Set
\begin{equation}\label{3.16}
b(\m)=\prod_{p|\prod_Am_A}\left(1-\frac1p\right)^{k-1}
\end{equation}
Then (\ref{3.15}) implies (the summation below is over $\m,\l$ satisfying (\ref{3.7}), (\ref{3.9})-(\ref{3.12}))
$$
P_k(x) \geq \sum_{\m,\l}\P(\l,\m) \gg (\log x)^{-k}\sum_{\m}\frac{b(\m)}{\prod_{i=1}^kM_i}\sum_{l_1,\ldots,l_{k-1}}E^{l_1+\ldots+l_{k-1}}\sum_{l_k}E^{l_k};
$$
by (\ref{3.9}), the innermost sum here is over $l_k$ with
$$
E^{l_k} \leq \frac{\prod_{\substack{2\leq j\leq k\\ j \mbox{\tiny \, even}}}{\prod_{C\subseteq [k], |C|=j}m_C}}{\prod_{\substack{3\leq i\leq k\\ i \mbox{\tiny   \,odd}}}{\prod_{B\subseteq [k], |B|=i}m_B}} E^{-(l_1+\ldots+l_{k-1})}x
$$
By (\ref{3.12}), the right side is at least $1$, and thus 
$$
\sum_{l_1,\ldots,l_{k-1}}E^{l_1+\ldots+l_{k-1}}\sum_{l_k}E^{l_k} \gg x(\log x)^{k-1}\frac{\prod_{\substack{2\leq j\leq k\\ j \mbox{\tiny \, even}}}{\prod_{C\subseteq [k], |C|=j}m_C}}{\prod_{\substack{3\leq i\leq k\\ i \mbox{\tiny   \,odd}}}{\prod_{B\subseteq [k], |B|=i}m_B}}.
$$
It follows that 
\begin{equation}\label{3.17}
P_k(x) \gg \frac{x}{\log x}\sum_{\m}\frac{\prod_{\substack{2\leq j\leq k\\ j \mbox{\tiny \, even}}}{\prod_{C\subseteq [k], |C|=j}m_C}}{\prod_{\substack{3\leq i\leq k\\ i \mbox{\tiny   \,odd}}}{\prod_{B\subseteq [k], |B|=i}m_B}}\frac{1}{\prod_{i=1}^kM_i}b(\m).
\end{equation}
Using Lemma \ref{lem2.1} to rewrite the least common multiples $M_i$, we get
$$
\prod_{i=1}^kM_i=\frac{\prod_{\substack{2\leq j\leq k\\ j \mbox{\tiny \, even}}}{\prod_{C\subseteq [k], |C|=j}m_C^j}}{\prod_{\substack{3\leq i\leq k\\ i \mbox{\tiny   \,odd}}}{\prod_{B\subseteq [k], |B|=i}m_B^i}}
$$
Thus by (\ref{3.17}), 
\begin{equation}\label{3.18}
P_k(x) \gg \frac{x}{\log x}\sum_{\m}\frac{b(\m)}{D(\m)},
\end{equation}
where the summation is over $\m$ satisfying (\ref{3.7}), (\ref{3.10}), (\ref{3.11}), and
\begin{equation}\label{3.19}
D(\m)=\frac{\prod_{\substack{2\leq j\leq k\\ j \mbox{\tiny \, even}}}{\prod_{C\subseteq [k], |C|=j}m_C^{j-1}}}{\prod_{\substack{3\leq i\leq k\\ i \mbox{\tiny   \,odd}}}{\prod_{B\subseteq [k], |B|=i}m_B^{i-1}}}.
\end{equation}
Now we simplify this expression.
 
\begin{lem}\label{lem3.4}
Let ${k\choose 2}$-tuple $\m=(m_A: A\subseteq [k], |A|=2)$ satisfy (\ref{3.7}). Then
\begin{equation}\label{3.20}
D(\m)=\lcm\{m_A: |A|=2\}.	
\end{equation}
Besides, for any prime $p$,
\begin{equation}\label{3.21}
\#\{\m: (\ref{3.7}), D(\m)=p\}=2^k-k-1.
\end{equation}
\end{lem}

\begin{proof}
As we saw, if $\m$ obeys (\ref{3.7}), then it occurs as $(\gcd(m_i,m_j): 1\leq i<j\leq k)$ for some $k$-tuple $(m_1,\ldots,m_k)$. Suppose that, for a prime $p$, we have $p^{\a_s}||m_s$ with nonnegative $\a_s$; without loss of generality, we may assume that $\a_1\leq \ldots \leq \a_k$. By (\ref{3.19}), $p^{\nu}||D(\m)$ is equivalent to 
\begin{equation*} 
\sum_{j=2}^{k}(-1)^{j}(j-1)\sum_{s=1}^{k-j+1}{k-s \choose j-1}\a_s=\nu.
\end{equation*}
The left side is
$$
\sum_{s=1}^{k-1}\a_s\sum_{j=1}^{k-s}(-1)^{j-1}j{k-s \choose j}=\a_{k-1},
$$
so that $p^{\a_{k-1}}||D(\m)$, and (\ref{3.20}) follows.

Now we turn to the second claim, which corresponds to the case $\nu=1$. We can argue similarly and assume that $(m_1,\ldots,m_k)$ is equal, up to a permutation, to one of the $k$-tuples
$$
(p^{\a_k},p,p^{\a_{k-2}},\ldots,p^{\a_1}),  
$$
where $\a_k\geq1$ and $\a_j\in\{0,1\}$ for all $1\leq j\leq k-2$. We may assume that $\a_k=1$, since it does not influence the resulting ${k\choose 2}$-tuple $\m$. There are exactly
$$
2^k-{k \choose 0}-{k \choose 1}=2^k-k-1
$$
$k$-tuples of the above form (since we just need to count those which have at least two coordinates which are equal to $1$). We claim that these $k$-tuples produce distinct ${k\choose 2}$-tuples $\m=(m_A: |A|=2)$; since each of them satisfies $D(\m)=p$, the claim (\ref{3.21}) will follow. So, it suffices to prove the injectivity of the map $\psi \colon X \to Y$,
where 
\begin{align*}
& X=\bigg\{\ba=(\a_1,\ldots,\a_k)\in \{0,1\}^k: \sum_{i=1}^k\a_i\geq2\bigg\}, \\
& Y=\big\{(\a_{ij}): \a_{ij}\in\{0,1\}, 1\leq i<j\leq k\big\},
\end{align*}
and $\psi$ is defined by 
$$
\psi(\ba)=\big(\min\{\a_i,\a_j\}: 1\leq i<j\leq k\big).
$$
Suppose that we know $\psi(\ba)$; since there are at least two of $\a_i$'s which are equal to $1$, we see that at least one of $\a_{ij}$ is equal to $1$. If this is $\a_{i_1i_2}$, then we must have $\a_{i_1}=\a_{i_2}=1$. Now one can easily recover all the other $\a_s$: clearly, $\a_s=1$ iff $\a_{i_1s}=1$. Thus, $\psi$ is injective and the number of ${k\choose 2}$-tuples $\m$ with $D(\m)=p$ is equal to $2^k-k-1$, as desired.
\end{proof}

\medskip 

The inequality (\ref{3.18}) can be rewritten as
\begin{equation}\label{3.22} 
P_k(x) \gg \frac{x}{\log x}\sum_{m}\frac{\mu^2(m)}{m} \sum_{\substack{\m: D(\m)=m\\ (\ref{3.7}), (\ref{3.10}), (\ref{3.11})}}b(\m)
\end{equation}
Define 
$$
g_k(m)=\mu^2(m)\sum_{\substack{\m: D(\m)=m\\ \m \mbox{ \tiny{obeys} } (\ref{3.7})}}b(\m).
$$
The sum in the right side of (\ref{3.22}) is closely related to the sum
$$
\sum_{x_0^{1/4}<m\leq x_0^{1/2}}\frac{g_k(m)}{m},
$$
which one can compute using the following classical estimates for the average value of a multiplicative function.

\begin{lem}\label{lem3.5}
Let $f$ be a nonnegative multiplicative function such that 
	
(a) there is $\theta>1$ with 
\begin{equation*}
\sum_{p\leq x}\sum_{\nu\geq1}\frac{f(p^{\nu})^{\theta}}{p^{\nu}} \ll \log_2 x; 
\end{equation*}
	
(b)	for some $\kappa\geq0$, one has  
\begin{equation*}
\sum_{p\leq x}\frac{f(p)\log p}{p}= \kappa \log x+O(1).
\end{equation*}
Then
$$
\sum_{n\leq x}\frac{f(n)}{n}=a(f,\kappa)(\log x)^{\kappa}+O\left((\log x)^{\kappa-1}\right), \quad\quad x\geq2, 
$$
for some constant $a(f,\kappa)>0$.
\end{lem}

\begin{proof}
This is well-known and due to Wirsing \cite{Wir}; see also \cite[Exercise 14.6]{D}.	
\end{proof}

Clearly, the above function $g_k$ is multiplicative and supported on square-free numbers. Also, $g_k(p)=(2^k-k-1)(1+O(1/p))$ due to Lemma \ref{lem3.4} and the definition (\ref{3.16}) of $b(\m)$, and thus Lemma \ref{lem3.5} with $\kappa=2^k-k-1$ implies 
$$
\sum_{x_0^{1/4}<m\leq x_0^{1/2}}\frac{g_k(m)}{m} \asymp (\log x)^{2^k-k-1}
$$
Also, since the number $Q_0$ is either equal to one or has at most $k$ prime divisors which are all $\gg\log x$, we have
$$
\sum_{\substack{x_0^{1/4}<m\leq x_0^{1/2}\\ \gcd(m,Q_0)>1}}\frac{g_k(m)}{m} \ll  (\log x)^{2^k-k-2},
$$
and, hence,
$$
\sum_{\substack{x_0^{1/4}<m\leq x_0^{1/2}\\ \gcd(m,Q_0)=1}}\frac{g_k(m)}{m} \gg  (\log x)^{2^k-k-1}.
$$
It follows that
$$
\sum_{x_0^{1/4}<m\leq x_0^{1/2}}\frac{\mu^2(m)}{m} \sum_{\substack{\m: D(\m)=m\\ (\ref{3.7}), (\ref{3.10})}}b(\m) \gg  (\log x)^{2^k-k-1}.
$$
It remains to take care of the condition (\ref{3.11}). Since $x_0^{1/4}<D(\m)\leq x_0^{1/2}$ and 
$$
D(\m)=\lcm\{m_A: |A|=2\} \leq M_1\cdot\ldots\cdot M_k,
$$
we see that there is $1\leq i\leq k$ with $M_i>D(\m)^{1/k}\geq x_0^{1/(4k)}$. Define
$$
S_i=\sum_{x_0^{1/4}<m\leq x_0^{1/2}}\frac{\mu^2(m)}{m} \sum_{\substack{\m: D(\m)=m\\ (\ref{3.7}), (\ref{3.10})\\ M_i>x_0^{1/(4k)}}}b(\m)
$$
Then
$$
\sum_{i=1}^kS_i\geq \sum_{x_0^{1/4}<m\leq x_0^{1/2}}\frac{\mu^2(m)}{m} \sum_{\substack{\m: D(\m)=m\\ (\ref{3.7}), (\ref{3.10})}}b(\m) \gg  (\log x)^{2^k-k-1}.
$$
By symmetry, all $S_i$ are equal. Therefore, from (\ref{3.22}) we have
$$
P_k(x) \gg \frac{x}{\log x}S_1 \gg x(\log x)^{2^k-k-2},
$$
which is the lower bound in (\ref{2.1}). This completes the proof of the lower bound in Theorem \ref{th1}.


\section{The upper bound} \label{sec4}

We induct on $k$. Fix $k\geq4$ and assume that
$$
P_{k-1}(x)\ll x(\log x)^{2^{k-1}-k-1}
$$
for all $x$. This allows us to count $k$-primes $(p_1,\ldots,p_k)$ with $[p_1-1,\ldots,p_k-1]\leq x$ and distinct $p_j$, while estimating $P_k(x)$.

All subsets of $[k]$ considered below are assumed to be nonempty. Let a $k$-tuple of primes $\p=(p_1,\ldots,p_k)$ be counted in $P_k(x)$, and recall the quantities $u_S(\p)$ defined in (\ref{2.7}). Then
\begin{equation}\label{4.1}
[p_1-1,\ldots,p_k-1]=\prod_{S\subseteq [k]}u_S(\p) \leq x,
\end{equation}
and recall also that
\begin{equation}\label{4.2}
\gcd\left(\prod_{S\supseteq S_1}u_S(\p),\prod_{S\supseteq S_2}u_S(\p)\right)=\prod_{S\supseteq S_1\cup S_2}u_S(\p)
\end{equation}
for all $S_1,S_2$. We have
$$
P_k(x) \leq \sum_{T\subseteq[k]}P_{k,T}(x),
$$
where $P_{k,T}(x)$ counts the number of $k$-tuples with $u_T(\p)=\max_{S}u_S(\p)$, so that it is enough to establish the bound    
\begin{equation}\label{4.3}
P_{k,T}(x) \ll x(\log x)^{2^k-k-1}	
\end{equation}
for each $T\subseteq[k]$. Fix $T$ and let $|T|=k-\nu$ for some $0\leq \nu\leq k-1$. By symmetry, we can assume that $T=T_0:=\{\nu+1,\ldots,k\}$. Since there are $2^k-1$ factors in (\ref{4.1}), we have 
$$
\prod_{S\neq T_0}u_S(\p) \leq x^{1-1/(2^k-1)}.
$$
Assume that $\nu>0$ for now. Let $V_0=1$ and take a prime $p_1\leq x$. Recall the functions $\tau_l$ defined in (\ref{2.10}). There are $\tau_{2^{k-1}}(p_1-1)$ factorizations $p_1-1=\prod_{S\ni 1}u_S$ with positive integer $u_S$. We fix any of them such that the `current' set $\{u_S: S\ni 1\}$ of $u_S$ satisfies (\ref{4.2}), and set 
$$
V_1=\prod_{S\ni 1,2} u_S.
$$ 
Next, choose a prime $p_2$ which is equal to $1$ modulo $V_1$ and fix any of $\tau_{2^{k-2}}((p_2-1)/V_1)$ factorizations $(p_2-1)/V_1=\prod_{\min S=2}u_S$ such that the set of numbers $\{u_S: \min S\leq 2\}$ satisfies (\ref{4.2}). We define 
$$
V_2=\prod_{\substack{S\ni 3\\\min S\leq 2}} u_S
$$
and continue this procedure. More precisely, for each $1\leq i\leq \nu-1$, once primes $p_1,\ldots,p_i$ and all $u_S$ with $\min S\leq i$ and (\ref{4.2}) are chosen, we take a prime $p_{i+1}\equiv 1\pmod {V_i}$ and fix any of $\tau_{2^{k-i}}((p_{i+1}-1)/V_i)$ factorizations $(p_{i+1}-1)/V_i=\prod_{\min S= i+1}u_S$ such that the set $\{u_S\}_{\min S\leq i+1}$ obeys (\ref{4.2}), and write
$$
V_{i+1}=\prod_{\substack{S\ni i+1\\\min S\leq i}} u_S.
$$
After $\nu$ steps, the procedure terminates. Define the number
\begin{equation}\label{4.4}
G=\prod_{S:\, \min S\leq \nu}u_S=\prod_{i=1}^{\nu}\frac{p_i-1}{V_{\nu-1}} 
\end{equation}
(the last equality holds because a number $u_S$ with $\min S=i\in\{1,...,\nu\}$ is a factor of \\$(p_i-1)/V_{i-1}$). We finally choose $u_S$ for each $S$ with $\min S>\nu$, $S\neq T_0$, such that 
\begin{equation}\label{4.5}
\prod_{\min S>\nu, \, S\neq T_0}u_S\leq x^{1-1/(2^k-1)}/G
\end{equation}
and all the numbers $u_S, S\neq T_0,$ satisfy (\ref{4.2}). For $j\geq \nu+1$, let
\begin{equation}\label{4.6}
a_j=\prod_{S\ni j, \, S\ne T_0}u_S. 
\end{equation}
Then $u_{T_0}\leq x/\prod_{S\neq T_0}u_S$ must be such that $a_ju_{T_0}+1$ is a prime (actually just $p_j$) for each $\nu+1 \leq j\leq k$. Since we assume that all primes $p_1,\ldots,p_k$ are distinct, we can think that the numbers $a_j$ are distinct for $\nu+1\leq j\leq k$. Also, by (\ref{4.5}), $\log\big(x/\prod_{S\neq T_0}u_S\big) \gg \log x$. By Brun's sieve (see, for instance, \cite[Theorem 2.3]{HR}) and (\ref{4.4}), there are at most
\begin{align*}
&\ll \frac{x}{\prod_{S\neq T_0}u_S\cdot (\log x)^{k-\nu}}\left(\frac{\Delta}{\varphi(\Delta)}\right)^{k-\nu} =\\ & \prod_{i=1}^{\nu}\frac{1}{(p_i-1)/V_{i-1}}\cdot\frac{x}{\prod_{\substack{S\neq T_0\\\min S>\nu}}u_S\cdot (\log x)^{k-\nu}}\left(\frac{\Delta}{\varphi(\Delta)}\right)^{k-\nu} 
\end{align*}
options for $u_{T_0}$, where
\begin{equation}\label{4.7}
\Delta=\prod_{j=\nu+1}^ka_j\prod_{\nu<i<j\leq k}|a_i-a_j|
\end{equation}
(note that $\Delta\neq0$ since the numbers $a_j$ are distinct). Therefore,
\begin{multline}\label{4.8}
P_{k,{T_0}}(x) \ll 
\sum_{p_1}\frac{\tau_{2^{k-1}}(p_1-1)}{p_1-1} \sum_{p_2}\frac{\tau_{2^{k-2}}((p_2-1)/V_1)}{(p_2-1)/V_1}\ldots\sum_{p_{\nu}}\frac{\tau_{2^{k-\nu}}((p_{\nu}-1)/V_{\nu-1})}{(p_{\nu}-1)/V_{\nu-1}} \times \\	
\times \sum_{u_S: \min S>\nu, \, S\neq T_0}\frac{x}{\prod_{\substack{S\neq T_0, \\ \min S>\nu }}u_S\cdot (\log x)^{k-\nu}}\left(\frac{\Delta}{\varphi(\Delta)}\right)^{k-\nu}, 
\end{multline}
where the outer summations are over $p_j\leq x$ with $p_j\equiv 1\pmod{V_{j-1}}$ for all $j=1,\ldots,\nu$. Recall that we assumed $\nu>0$; if $\nu=0$, the bound (\ref{4.8}) should be understood as excluding the summations over $p_1,\ldots,p_{\nu}$. 

Further, by the definitions (\ref{4.4}), (\ref{4.5}), (\ref{4.7}) of $G, a_{\nu+1},\ldots,a_k$ and $\Delta$,
$$
\left(\frac{\Delta}{\varphi(\Delta)}\right)^{k-\nu}\leq \left(\frac{G}{\varphi(G)}\right)^k\left(\frac{\Delta'}{\varphi(\Delta')}\right)^k\leq\prod_{i=1}^{\nu}\left(\frac{(p_i-1)/V_{i-1}}{\varphi((p_i-1)/V_{i-1})}\right)^k\left(\frac{\Delta'}{\varphi(\Delta')}\right)^k,
$$
where
\begin{equation}\label{4.9}
\Delta'=\prod_{\substack{S: \, \min S>\nu, \\ S\neq T_0}}u_S \cdot \prod_{\nu<i<j\leq k}\left(\prod_{\substack{S\ni i\\ S\notni j}}u_S-\prod_{\substack{S\ni j\\ S\notni i}}u_S\right) 
\end{equation}
(note that, for any $\nu<i<j\leq k$, the condition $S\ni i$, $S\notni j$ implies $S\neq T_0$, and that the numbers $\prod_{S\ni i,\, S\notni j}u_S$ and $\prod_{S\ni j, \, S\notni i}u_S$ are coprime due to (\ref{4.2}) applied with $S_1=\{i\}$ and $S_2=\{j\}$). We highlight that $\Delta'$ depends not only on $u_S$ with $\min S>\nu$, but also on $u_S$ with $\min S\leq \nu$ (which appeared in the factorizations of $p_1-1, (p_2-1)/V_1,\ldots, (p_{\nu}-1)/V_{\nu-1})$ we chosen). Let 
\begin{equation}\label{4.10}
f_i(n)=\frac{\tau_{2^{k-i}}(n)n^{k-1}}{\varphi(n)^k}, \quad i=1,\ldots, \nu.
\end{equation}
Then (\ref{4.8}) can be rewritten as
\begin{multline}\label{4.11}
P_{k,T_0}(x) \ll \frac{x}{(\log x)^{k-\nu}}\sum_{p_1}f_1(p_1-1) \sum_{p_2}f_2((p_2-1)/V_1)\ldots\sum_{p_{\nu}}f_{\nu}((p_{\nu}-1)/V_{\nu-1}) \times \\	\times \sum_{u_S:\, \min S>\nu, \, S\neq T_0}\frac{1}{\prod_{\substack{S\neq T_0, \\ \min S>\nu }}u_S}\left(\frac{\Delta'}{\varphi(\Delta')}\right)^k. 	
\end{multline}
Since all $u_S$ are at most $x$ and there are exactly $2^{k-\nu}-1$ subsets $S\subseteq [k]$ with $\min S>\nu$ and $S\neq T_0$, the innermost sum above would be 
$$
\ll (\log x)^{2^{k-\nu}-1},
$$
if we ignore the factor $(\Delta'/\varphi(\Delta'))^k$. Our next goal is to show that the same bound holds for the innermost sum itself. To do so, we need the following (actually straightforward) generalization of \cite[Lemma 4]{FP}.

\begin{lem}\label{lem4.1}
Uniformly for numbers $u_S\leq x$ with $\min S\leq \nu$,
\begin{equation}\label{4.12}
\sum_{\substack{S: \, \min S>\nu, \\ S\neq T_0}}\sum_{u_S\leq x}\frac{1}{\prod_{\substack{S\neq T_0, \\ \min S>\nu }}u_S}\prod_{r|\Delta'}\left(1+\frac1r\right)^k\ll (\log x)^{2^{k-\nu}-1}, 		
\end{equation}	
where the summation is over $u_S\leq x$ such that (\ref{4.2}) holds and $\Delta'\neq0$.
\end{lem}

\begin{proof}
Since the product over the prime divisors of $\Delta'$ is at most $\ll(\log_2x)^k$ uniformly for all $u_S\leq x$, the claim holds if at least one of the variables $u_S$ with $\min S>\nu$ is restricted to integers up to $x^{1/(\log_2x)^k}$. Thus, we can assume that all $u_S$ go over the interval $(x^{1/(\log_2x)^k},x]$. Further, since each positive integer $n\leq x$ has at most $O(\log x)$ prime divisors, it is easy to see that
$$
\prod_{r|n: \, r>(\log x)^{1/2}}\left(1+\frac1r\right)^k \ll 1
$$	
uniformly for $n\leq x$. Then
$$
\prod_{r|n}\left(1+\frac1r\right)^k \ll \prod_{r|\gcd(n,P)}\left(1+\frac{k}{r}\right)=\sum_{d|\gcd(n,P)}\frac{k^{\o(d)}}{d}
$$
where 
$$
P=\prod_{r\leq (\log x)^{1/2}}r.
$$
Thus, to prove (\ref{4.12}), it suffices to show that 
\begin{equation}\label{4.13}
\sum_{\substack{x^{1/(\log_2x)^k}<u_S\leq x\\\min S>\nu, \, S\neq T_0}}\frac{1}{\prod_{\substack{S\neq T_0, \\ \min S>\nu }}u_S}\sum_{d|\gcd(P,\Delta')}\frac{k^{\o(d)}}{d}\ll (\log x)^{2^{k-\nu}-1}.
\end{equation}
Since each $d|P$ is at most $P=\exp\left((1+o(1))(\log x)^{1/2}\right)<x^{1/(\log_2x)^k}$ for large enough $x$, the left side here is 
\begin{equation}\label{4.14}
\ll \sum_{d|P}\frac{k^{\o(d)}}{d}\sum_{\substack{d<u_S\leq x\\\min S>\nu, \, S\neq T_0:\\ \Delta'\equiv 0\pmod d}}\frac{1}{\prod_{\substack{S\neq T_0, \\ \min S>\nu }}u_S}.
\end{equation}
Let $w=2^{k-\nu}-1$ for short. Fix $d|P$ and recall the definition (\ref{4.9}) of $\Delta'$. For a prime divisor $r$ of $d$, we have $\Delta'\equiv 0\pmod p$ if either some of $u_S\equiv0\pmod p$ (here $\min S>\nu, S\neq T_0$), or $\prod_{S \ni i, S\notni j}u_S \equiv \prod_{S \ni j, S\notni i}u_S \pmod p$ for some $\nu<i<j\leq k$. In the latter case, it is impossible that both products are equal to zero modulo $p$, since they are coprime by (\ref{4.2}). Therefore, in any of the mentioned cases we have one of the $w$ variables $u_S$ being determined modulo $p$ in terms of the remaining ones. Since $d$ is square-free, the number of $w$-tuples of variables $u_S$ modulo $d$ with $\Delta'\equiv 0\pmod d$ is at most $L^{\o(d)}d^{w-1}$ provided that the constant $L$ is large enough depending on $k$. For a fixed $w$-tuple of residues modulo $d$ we have here, the sum of $\prod_{S\neq T_0, \min S>\nu }u_S^{-1}$ in this class is $\ll (\log x)^{w}d^{-w}$ uniformly over the $w$-tuples modulo $d$ under consideration. It follows that the right side of (\ref{4.14}) is 
$$
\ll (\log x)^w\sum_{d|P}\frac{(Lk)^{\o(d)}}{d^2}  \leq (\log x)^w\prod_{r}\left(1+\frac{Lk}{r^2}\right) \ll (\log x)^w,
$$
and (\ref{4.13}) follows. This completes the proof.
\end{proof}

The estimate (\ref{4.11}), together with Lemma \ref{lem4.1}, implies
\begin{equation}\label{4.15}
P_{k,T_0}\ll x(\log x)^{2^{k-\nu}-k+\nu-1}\sum_{p_1}f_1(p_1-1) \sum_{p_2}f_2((p_2-1)/V_1)\ldots\sum_{p_{\nu}}f_{\nu}((p_{\nu}-1)/V_{\nu-1}), 
\end{equation}
and in particular the desired bound (\ref{4.3}) follows in the case $\nu=0$. If $\nu\geq1$, it remains to estimate each of the $\nu$ sums above. Let $\M$ be the collection of nonnegative-valued multiplicative functions $f$ satisfying
the conditions

(i) there is a constant $A_1>0$ such that $f(p^i) \leq A_1^i$ for all prime powers $p^i$;

(ii) for every $\e>0$, there is a constant $A_2(\e)>0$ such that $f(n) \leq A_2(\e)n^{\e}$ for all $n\geq1$.

\medskip 

We need the following result, which is Theorem 1.1 of \cite{Pol}.

\begin{theorem}[Brun’s upper bound sieve for multiplicative functions]\label{th4.2}
Let $f \in \M$ and $s$ be a nonnegative integer. For each prime $p \leq x$, let $\E_p$ be a union of $v(p)$ nonzero residue classes modulo $p$, where we suppose that each $v(p)\leq s$. Let
$$
\S=\bigcap_{p\leq x}\E^c_p,
$$ 
that is, $\S$ is the set of all positive integers not belonging to any $\E_p$. If $x$ is sufficiently large, then
$$
\sum_{\substack{n\leq x\\n\in \S}} f(n) \ll \frac{x}{\log x}\exp\left(\sum_{p\leq x}\frac{f(p)-v(p)}{p}\right)
$$
The implied constant here, as well as the threshold for ``sufficiently large'', depends only on $s$, the constant $A_1$ in (i) above, and the function $A_2(\e)$ in (ii).	
\end{theorem}

We write 
\begin{align*} 
\sum_{\substack{p_{\nu}\leq x \\ p_{\nu}\equiv 1\pmod{V_{\nu-1}}}}\tau_{2^{k-\nu}}\left(\frac{p_{\nu}-1}{V_{\nu-1}}\right)\left(\frac{(p_{\nu}-1)/V_{\nu-1}}{\varphi((p_{\nu}-1)/V_{\nu-1})}\right)^k\leq\sum_{\substack{n\leq x/V_{\nu-1}: \\ nV_{\nu-1}+1 \mbox{\tiny{ prime }} }}\tau_{2^{k-\nu}}(n)\left(\frac{n}{\varphi(n)}\right)^k. 
\end{align*}
Recall that each $V_i$ is at most $x^{1-1/(2^k-1)}$ and let $\e>0$ be small enough. If $nV_{\nu-1}+1$ is prime for some $(x/V_{\nu-1})^{1/2}<n\leq x/V_{\nu-1}$, then it does not have prime divisors less than $x^{1/2}V_{\nu-1}^{1/2}$. Define the set $\S$ via $\E_p=\varnothing$ for $p|V_{\nu-1}$ or $x^{1/2}V_{\nu-1}^{1/2}<p\leq x/V_{\nu-1}$, and $\E_p=\{-V_{\nu-1}^{-1} \pmod p\}$ otherwise. Applying Theorem \ref{th4.2} with $s=1$, we see that the right side above is
$$
(x/{V_{\nu-1}})^{1/2+\e}+\sum_{\substack{n\leq x/V_{\nu-1} \\ n\in \S }}\tau_{2^{k-\nu}}(n)\left(\frac{n}{\varphi(n)}\right)^k \ll \frac{x(\log x)^{2^{k-\nu}-2}}{\varphi(V_{\nu-1})}.
$$
Thus, due to the definition (\ref{4.10}) of the functions $f_i$, 
$$
\sum_{\substack{p_{\nu}\leq x \\ p_{\nu}\equiv 1\pmod{V_{\nu-1}}}}f_{\nu}\big((p_{\nu}-1)/V_{\nu-1}\big) \ll (\log x)^{2^{k-\nu}-1}\frac{V_{\nu-1}}{\varphi(V_{\nu-1})}.
$$
Now $V_{\nu-1}$ is a factor of $(p_{\nu-1}-1)/V_{\nu-2}$ and we can argue similarly to get
$$
\sum_{\substack{p_{\nu-1}\leq x \\ p_{\nu-1}\equiv 1\pmod{V_{\nu-2}}}}f_{\nu-1}\left(\frac{p_{\nu-1}-1}{V_{\nu-2}}\right)\frac{(p_{\nu-1}-1)/V_{\nu-2}}{\varphi(p_{\nu-1}-1)/V_{\nu-2}}  \ll (\log x)^{2^{k-\nu+1}-1}\frac{V_{\nu-2}}{\varphi(V_{\nu-2})}.
$$
After $\nu-2$ more such steps, we arrive at the estimate
\begin{align*} 
& \sum_{p_1}f_1(p_1-1) \sum_{p_2}f_2((p_2-1)/V_1)\ldots\sum_{p_{\nu}}f_{\nu}((p_{\nu}-1)/V_{\nu-1})\ll \\ & (\log x)^{2^{k-1}+2^{k-2}+...+2^{k-\nu}-\nu}=(\log x)^{2^k-2^{k-\nu}-\nu}.
\end{align*}
From this and (\ref{4.15}) we conclude the desired estimate (\ref{4.3}), and the upper bound in Theorem \ref{th1} follows.


\section{The distribution of $\o^*$} \label{sec5}

For $1\leq y\leq x$, let 
$$
N(x,y)=\#\{n\leq x: \o^*(n)\geq y\}.
$$ 
In \cite{FP}, the bounds
$$
\frac{x}{y^{c_0\log_2 y}} \ll N(x,y) \ll \frac{x\log y}{y}
$$	
were proved (the lower bound is essentially a consequence of \cite[Proposition 10]{APM}), where $c_0>0$ is an absolute constant. The goal of this section is to provide more precise estimates. First of all, Markov's inequality and Theorem \ref{th1} easily imply the following.  	

\begin{prop} \label{prop5.1}
For any fixed positive integer $k\geq2$,
\begin{equation}\label{5.1} 
N(x,y)\ll \min\bigg\{\frac{x\log y}{y}, \, \min_{2\leq j\leq k}\frac{x(\log x)^{2^j-j-1}}{y^j}\bigg\}.
\end{equation}
\end{prop}

Clearly, the first bound is better for $y\ll \frac{\log x}{\log_2 x}$, and, for $j\geq2$, the bound $N(x,y) \ll x(\log x)^{2^j-j-1}y^{-j}$ is the best one in the range  $(\log x)^{2^{j-1}-1} \ll y\ll (\log x)^{2^j-1}$. In fact, simple `concentration' arguments show that these estimates are sharp up to the factor of $\log_2x$ at least for some values of $y$.

\begin{prop} \label{prop5.2}
i) There are values $y$ with $\exp(c(\log_2x)^{1/2}) \ll y\ll \frac{\log x}{\log_2x}$, where $c>0$ is a small absolute constant, such that 
$$
N(x,y) \gg \frac{x}{y}.
$$

ii) There are values $y$ with $\frac{\log x}{\log_2x}\ll y\ll (\log x)^3$ with
$$
N(x,y) \gg \frac{x\log x}{y^2\log_2x}.
$$

iii) For any fixed $j\geq3$, there are values $y$ with $(\log x)^{2^{j-1}-1}\ll y\ll (\log x)^{2^j-1}$ such that 
$$
N(x,y) \gg \frac{x(\log x)^{2^j-j-1}}{y^j\log_2x}.
$$
\end{prop}

\smallskip 

\begin{proof}
We begin with the first claim. First of all, 
\begin{equation}\label{5.2} 
\sum_{k\geq 0} e^k\,\#\{n\leq x: e^k< \o^*(n)\leq e^{k+1}\}\geq e^{-1}\sum_{n\leq x}\o^*(n) \geq 0.3x\log_2 x
\end{equation}
by (\ref{1.2}). For $k\leq c(\log_2x)^{1/2}$, from the first estimate in (\ref{5.1}) we get 	\begin{align}\label{5.3} 
& \sum_{k\leq c(\log_2x)^{1/2}} e^k\,\#\{n\leq x: e^k< \o^*(n)\leq e^{k+1}\} \ll \\ & x\sum_{k\leq c(\log_2x)^{1/2}}k \ll c^2x\log_2 x < 0.1x\log_2 x \nonumber 
\end{align}
provided that $c>0$ is small enough. Now let $u=\log_2x-\log_3x+C$ with $C>0$ large enough. Using the second estimate in (\ref{5.1}) with $j=2$, we get
\begin{equation}\label{5.4} 
\sum_{k>u} e^k\,\#\{n\leq x: e^k<\o^*(n)\leq e^{k+1}\} \ll x\log x\sum_{k>u}e^{-k} \leq 0.1x\log_2x.
\end{equation}
Combining (\ref{5.2})-(\ref{5.4}), we see that
$$
\sum_{c(\log_2x)^{1/2} < k\leq \log_2x}e^k\,\#\{n\leq x: e^k< \o^*(n)\leq e^{k+1}\} \asymp x\log_2x.
$$
By the pigeonhole principle, there is $k$ with $c(\log_2x)^{1/2} < k\leq u$ such that
$$
e^k\,\#\{n\leq x: e^k< \o^*(n)\leq e^{k+1}\} \gg x.
$$
and the first claim follows by taking $y=e^k$. 		
	
To establish the second claim, we similarly start with	
$$
\sum_{k\geq0}e^{2k}\#\{n\leq x: e^k< \o^*(n)\leq e^{k+1}\}\gg \sum_{n\leq x}\o^*(n)^2  \gg x\log x.
$$
Let $v=\log_2x-\log_3x-C$ with $C>0$ large enough. By the first estimate in (\ref{5.1}),
$$
\sum_{k\leq v}e^{2k}\#\{n\leq x: e^k< \o^*(n)\leq e^{k+1}\} \ll x\sum_{k\leq v}ke^{k} \ll e^{-C}x\log x. 
$$
and, by (\ref{5.1}) applied with $j=3$, 
\begin{align*} 
& \sum_{k>3\log_2x+C}e^{2k}\#\{n\leq x: e^k< \o^*(n)\leq e^{k+1}\} \ll \\
& x(\log x)^4\sum_{k>3\log_2x+C}e^{-k} \ll e^{-C}x\log x. 
\end{align*}
Thus,
$$
\sum_{v\leq k\leq 3\log_2x+C}e^{2k}\#\{n\leq x: e^k< \o^*(n)\leq e^{k+1}\} \asymp x\log x.
$$
By the pigeonhole principle, there is $k$ in the above range such that 
$$
e^{2k}\#\{n\leq x: e^k< \o^*(n)\leq e^{k+1}\} \gg \frac{x\log x}{\log_2x}.
$$
Taking $y=e^k$ this completes the proof. 

The last claim can be proved in the same manner. Fix $j\geq3$. By Theorem \ref{th1},
$$ 
\sum_{k\geq0}e^{jk}\#\{n\leq x: e^k< \o^*(n)\leq e^{k+1}\}\gg \sum_{n\leq x}\o^*(n)^j  \gg x(\log x)^{2^j-j-1}.
$$
Let $c_j=2^j-1$ for short and $C>0$ be large enough. By (\ref{5.1}) applied with $j-1$,
\begin{align*}
& \sum_{k\leq c_{j-1}\log_2x-C}e^{jk}\#\{n\leq x: e^k< \o^*(n)\leq e^{k+1}\} \ll \\& x(\log x)^{2^{j-1}-j}\sum_{k\leq c_{j-1}\log_2x-C}e^{k} \ll 
 e^{-C}x(\log x)^{2^j-j-1}. 
\end{align*}
and, by (\ref{5.1}) applied with $j+1$, 
\begin{align*} 
& \sum_{k>c_j\log_2x+C}e^{jk}\#\{n\leq x: e^k< \o^*(n)\leq e^{k+1}\} \ll \\
& x(\log x)^{2^{j+1}-j-2}\sum_{k>c_j\log_2x+C}e^{-k} \ll e^{-C}x(\log x)^{2^j-j-1}. 
\end{align*}
Thus,
$$
\sum_{c_{j-1}\log_2 x-C\leq k\leq c_j\log_2x+C}e^{jk}\#\{n\leq x: e^k< \o^*(n)\leq e^{k+1}\} \asymp x(\log x)^{2^j-j-1}.
$$
By the pigeonhole principle, there is $k$ in the above range such that 
$$
e^{jk}\#\{n\leq x: e^k< \o^*(n)\leq e^{k+1}\} \gg \frac{x(\log x)^{2^j-j-1}}{\log_2x}.
$$
This completes the proof. 
\end{proof}

\end{document}